\documentclass[a4paper,leqno,12pt]{amsart}
\usepackage{epsfig}
\usepackage{amsmath}
 \usepackage[pdftex]{hyperref}
\usepackage{amsfonts}
\usepackage{amsmath,amscd}
\usepackage{amssymb}
\usepackage{amsthm}
\usepackage{mathrsfs}
\usepackage[mathscr]{euscript}
\usepackage{graphicx}
\usepackage{xypic}
\usepackage{stackrel}
\newtheorem{definition}{Definition}[section]
\newtheorem{theorem}[definition]{Theorem}
\newtheorem{remark}[definition]{Remark}
\newtheorem{final Remarks}[definition]{Final Remarks}
\newtheorem{lemma}[definition]{Lemma}

\newtheorem{corollary}[definition]{Corollary}

\numberwithin{equation}{section}
\begin{document}
\title{Smooth Structures on a Fake Real Projective Space}
\vspace{2cm}
\author{Ramesh Kasilingam}

\email{rameshkasilingam.iitb@gmail.com  ; mathsramesh1984@gmail.com  }
\address{Statistics and Mathematics Unit,
Indian Statistical Institute,
Bangalore Centre, Bangalore - 560059, Karnataka, India.}
\date{}
\maketitle
\begin{abstract}
We show that the group of smooth homotopy $7$-spheres acts freely on the set of smooth manifold structures on a topological manifold $M$ which is homotopy equivalent to the real projective $7$-space. We classify, up to diffeomorphism, all closed manifolds homeomorphic to the real projective $7$-space. We also show that $M$ has, up to diffeomorphism, exactly $28$ distinct differentiable structures with the same underlying PL structure of $M$ and $56$ distinct differentiable structures with the same underlying topological structure of $M$. 
\end{abstract}
\section{Introduction}
Throughout this paper $M^m$ will be a closed oriented $m$-manifold and all homeomorphisms and diffeomorphisms are assumed to preserve orientation, unless otherwise stated. Let $\mathbb{R}\textbf{P}^n$ be real projective $n$-space. L$\acute{\rm o}$pez de Medrano \cite{Lop71} and C.T.C. Wall \cite{Wal68, Wal99} classified, up to PL homeomorphism, all closed PL
manifolds homotopy equivalent to  $\mathbb{R}\textbf{P}^n$ when $n>4$. This was extended to the topological category by Kirby-Siebenmann \cite[pp. 331]{KS77}. Four-dimensional surgery \cite{FQ90} extends the homeomorphism classification to dimension $4$.\\
\indent In this paper we study up to diffeomorphism all closed manifolds homeomorphic to $\mathbb{R}\textbf{P}^7$. Let $M$ be a closed smooth manifold homotopy equivalent to $\mathbb{R}\textbf{P}^7$. 
In section 2, we show that if a closed smooth manifold $N$ is PL-homeomorphic to $M$, then there is a unique homotopy $7$-sphere $\Sigma^7\in \Theta_7$ such that $N$ is diffeomorphic to $M\#\Sigma^7$, where $\Theta_7$ is the group of smooth homotopy spheres defined by M. Kervaire and J. Milnor in \cite{KM63}. In particular, $M$ has, up to diffeomorphism, exactly $28$ distinct differentiable structures with the same underlying PL structure of $M$.\\
In section 3, we show that if a closed smooth manifold $N$ is homeomorphic to $M$, then there is a unique homotopy $7$-sphere $\Sigma^7\in \Theta_7$ such that $N$ is diffeomorphic to either $M\#\Sigma^7$ or $\widetilde{M}\#\Sigma^7$, where $\widetilde{M}$ represents the non-zero concordance class of PL-structure on $M$. We also show that the group of smooth homotopy $7$-spheres $\Theta_7$ acts freely on the set of smooth manifold structures on a manifold $M$.
\section{Smooth structures with the same underlying PL structure of a fake real projective space}
We recall some terminology from \cite{KM63}:
\begin{definition}\rm \cite{KM63}
\begin{itemize}
\item[(a)] A homotopy $m$-sphere $\Sigma^m$ is a smooth closed manifold homotopy equivalent to the standard unit sphere $\mathbb{S}^m$ in $\mathbb{R}^{m+1}$.
\item[(b)]A homotopy $m$-sphere $\Sigma^m$ is said to be exotic if it is not diffeomorphic to $\mathbb{S}^m$.
\item[(c)] Two homotopy $m$-spheres $\Sigma^{m}_{1}$ and $\Sigma^{m}_{2}$ are said to be equivalent if there exists a diffeomorphism $f:\Sigma^{m}_{1}\to \Sigma^{m}_{2}$.
\end{itemize}
The set of equivalence classes of homotopy $m$-spheres is denoted by $\Theta_m$. The equivalence class of $\Sigma^m$ is denoted by [$\Sigma^m$]. M. Kervaire and J. Milnor \cite{KM63} showed that $\Theta_m$ forms a finite abelian group with group operation given by connected sum $\#$ except possibly when $m=4$ and the zero element represented by the equivalence class  of $\mathbb{S}^m$.
\end{definition}
\begin{definition}\rm
Let $M$ be a closed PL-manifold. Let $(N,f)$ be a pair consisting of a closed PL-manifold $N$ together with a homotopy equivalence $f:N\to M$. Two such pairs $(N_{1},f_{1})$ and $(N_{2},f_{2})$ are equivalent provided there exists a PL homeomorphism $g:N_{1}\to N_{2}$ such that $f_{2}\circ g$ is homotopic to $f_{1}$. The set of all such equivalence classes is denoted by $\mathcal{S}^{PL}(M)$.
\end{definition}
\begin{definition}\rm($Cat=Diff~~{\rm{or}}~~ PL$-structure sets)
Let $M$ be a closed $Cat$-manifold. Let $(N,f)$ be a pair consisting of a closed $Cat$-manifold $N$ together with a homeomorphism $f:N\to M$. Two such pairs $(N_{1},f_{1})$ and $(N_{2},f_{2})$ are concordant provided there exists a $Cat$-isomorphism $g:N_{1}\to N_{2}$ such that the composition $f_{2}\circ g$ is topologically concordant to $f_{1}$, i.e., there exists a homeomorphism $F: N_{1}\times [0,1]\to M\times [0,1]$ such that $F_{|N_{1}\times 0}=f_{1}$ and $F_{|N_{1}\times 1}=f_{2}\circ g$. The set of all such concordance classes is denoted by $\mathcal{C}^{Cat}(M)$.\\
We will denote the class in $\mathcal{C}^{Cat}(M)$ of $(N,f)$ by $[N,f]$. The base point of $\mathcal{C}^{Cat}(M)$ is the equivalence class $[M,Id]$ of $Id : M\to M$.\\
\indent We will also denote the class in $\mathcal{C}^{Diff}(M)$ of $(M^n\#\Sigma^n, \rm{Id})$ by $[M^n\#\Sigma^n]$. (Note that $[M^n\#\mathbb{S}^n]$ is the class of $(M^n, \rm{Id})$.)
\end{definition}
\begin{definition}\rm
Let $M$ be a closed PL-manifold. Let $(N,f)$ be a pair consisting of a closed smooth manifold $N$ together with a PL-homeomorphism $f:N\to M$. Two such pairs $(N_{1},f_{1})$ and $(N_{2},f_{2})$ are PL-concordant provided there exists a diffeomorphism $g:N_{1}\to N_{2}$ such that the composition $f_{2}\circ g$ is PL-concordant to $f_{1}$, i.e., there exists a PL-homeomorphism $F: N_{1}\times [0,1]\to M\times [0,1]$ such that $F_{|N_{1}\times 0}=f_{1}$ and $F_{|N_{1}\times 1}=f_{2}\circ g$. The set of all such concordance classes is denoted by $\mathcal{C}^{PDiff}(M)$.
\end{definition}
\begin{definition}\rm{
Let $M^m$ be a closed smooth $m$-dimensional manifold. The inertia group $I(M)\subset \Theta_{m}$ is defined as the set of $\Sigma \in \Theta_{m}$ for which there exists a diffeomorphism $\phi :M\to M\#\Sigma$.\\
The concordance inertia group $I_c(M)$ is defined as the set of all $\Sigma\in I(M)$ such that $M\#\Sigma$ is concordant to $M$.}
\end{definition}
The key to analyzing $\mathcal{C}^{Diff}(M)$ and $\mathcal{C}^{PDiff}(M)$ are the following results. 
\begin{theorem}{\rm (Kirby and Siebenmann, \cite[pp. 194]{KS77})}\label{kirby} 
There exists a connected $H$-space $Top/O$ such that there is a bijection between $\mathcal{C}^{Diff}(M)$ and $[M, Top/O]$ for any smooth manifold $M$ with $\dim M\geq5$. Furthermore, the concordance class of given smooth structure of $M$ corresponds to the homotopy class of the constant map under this bijection. 
\end{theorem}
\begin{theorem}{\rm (Cairns-Hirsch-Mazur, \cite{HM74})}\label{cairns} 
Let $M^m$ be a closed smooth manifold of dimension $m\geq 1$. Then there exists a connected $H$-space $PL/O$ such that there is a bijection between $\mathcal{C}^{PDiff}(M)$ and $[M, PL/O]$. Furthermore, the concordance class of the given smooth structure of $M$ corresponds to the homotopy class of the constant map under this bijection. 
\end{theorem}
Recall that the group of smooth homotopy $m$-spheres $\Theta_m$ acts on $\mathcal{C}^{Cat}(M^m)$, where $Cat=Diff$ or $PDiff$, by $$\Theta_m\times \mathcal{C}^{Cat}(M^m)\mapsto \mathcal{C}^{Cat}(M^m),~~([\Sigma], [N,f])\mapsto [N\#\Sigma,f]$$ where we regard the connected sum $N\#\Sigma$ as a smooth manifold with the same underlying topological space as $N$ and with smooth structure differing from that of $N$ only on an $m$-disc. 
\begin{theorem}\label{pdiff}
Let $M^7$ be a closed smooth $7$-manifold. Then the group $\Theta_7$ acts freely and transitively on $\mathcal{C}^{PDiff}(M)$.
In particular, $$\mathcal{C}^{PDiff}(M)=\left \{[M\#\Sigma]~~|~~ \Sigma\in \Theta_{7} \right \}.$$
\end{theorem}
\begin{proof}
For any degree one map $f_{M}:M^7\to \mathbb{S}^{7}$, we have a homomorphism
$$f_{M}^*:[\mathbb{S}^{7}, PL/O]\to [M^7, PL/O] $$ and in terms of the identifications
\begin{center}
$\Theta_7=[S^7, PL/O]$ and $\mathcal{C}^{PDiff}(M)=[M^7, PL/O]$
\end{center}
given by Theorem \ref{cairns}, $f_{M}^*$ becomes $[\Sigma]\mapsto [M\#\Sigma]$.  Therefore, to show that $\Theta_7$ acts freely and transitively on $\mathcal{C}^{PDiff}(M)$, it is enough to prove that $$f_{M}^*:[\mathbb{S}^{7}, PL/O]\to [M, PL/O]$$ is bijective. Let $M^{(6)}$ be the $6$-skeleton of a CW-decomposition for $M$ containing just one $7$-cell. Such a decomposition exists by \cite{Wal67}. Let $f_{M}:M\to M/M^{(6)}=\mathbb{S}^7$ be the collapsing map. Now consider the Barratt-Puppe sequence for the inclusion $i:M^{(6)}\hookrightarrow M$ which induces the exact sequence of abelian groups on taking homotopy classes $[-,PL/O]$
$$\cdots \to [SM^{(6)}, PL/O]{\to}[S^{7},PL/O] \stackrel{f^{*}_{M}}{\to} [M, PL/O] \stackrel{i^{*}}{\to}[M^{(6)}, PL/O]\cdots,$$ where $SM$ is the suspension of $M$. As $PL/O$ is $6$-connected \cite{KM63,CE68}, it follows that any map from $M^{(6)}$ to $PL/O$ is null-homotopic (see \cite[Theorem 7.12]{DK01}). Therefore $i^{*}: [M, PL/O] \to [M^{(6)}, PL/O]$ is the zero homomorphism and so $f_{M}^*:[\mathbb{S}^{7}, PL/O]\to [M, PL/O]$ is surjective. Since $M$ is of dimension $7$ and the fact that $PL/O$ is $6$-connected, then by Obstruction theory \cite{HM74} implies that the primary obstruction to a null-homotopy $$Ob_{PL/O}:[M^7,PL/O]\to H^{7}(M^7;\pi_7(PL/O))\cong \mathbb{Z}_{28}$$ is an isomorphism. Therefore the homomorphism $f_{M}^*:[\mathbb{S}^{7}, PL/O]\to [M, PL/O]$ is injective, proving the theorem.
\end{proof}
\begin{remark}\rm \indent\label{unique}
\begin{itemize}
\item[(i)] If $M^7$ is a closed smooth $7$-manifold, then by Theorem \ref{pdiff}, the concordance inertia group $I_c(M)=0$.
\item[(ii)] If $(N,f)$ represents an element in $\mathcal{C}^{PDiff}(M^7)$, then by Theorem \ref{pdiff} there is a homotopy $7$-sphere $\Sigma^7\in \Theta_7$ such that $N$ is diffeomorphic to $M\#\Sigma^7$. Moreover, if the inertia group $I(M)=0$, then such a homotopy $7$-sphere $\Sigma^7$ is unique.
\end{itemize}
\end{remark}
We now use the Eells-Kuiper $\mu$ invariant \cite{EK62, TZ13} to study the inertia group of smooth manifolds homotopy equivalent to $\mathbb{R}\textbf{P}^7$.
We recall the definition of the Eells-Kuiper $\mu$ invariant in dimension $7$. Let $M$ be a $7$-dimensional closed oriented spin smooth manifold such that the $4$-th cohomology group $H^4(M;\mathbb{R})$ vanishes. Since the spin cobordism group $\Omega^{Spin}_7$ is trivial \cite{Mil63}, $M$ bounds a compact oriented spin smooth manifold $N$. Then the first Pontrjagin class $p_1(N)\in H^4(N, M; \mathbb{Q})$ is well-defined. The Eells-Kuiper differential invariant $\mu(M)\in \mathbb{R}/\mathbb{Z}$
of $M$ is given by $$\mu(M)=\frac{p_1^2(N)}{2^7\times 7}-\frac{Sign(N)}{2^5\times 7}~~~{\rm mod(\mathbb{Z})},$$ where $p^2_1(N)$ denotes the corresponding Pontrjagin number and ${\rm Sign(N)}$ is the Signature of $N$. 
\begin{lemma}\label{inerreal}
Let $M$ be a closed smooth spin $7$-manifold such that $H^4(M;\mathbb{R})=0.$ Then $I(M)=0$.
\end{lemma}
\begin{proof}
Our assumption on $M$ and using the additivity of the Eells-Kuiper differential invariant $\mu$ with respect to connected sums, if $\Sigma\in I(M)$, then $$\mu(M)=\mu(M\#\Sigma)=\mu(M)+\mu(\Sigma).$$ Therefore $\mu(\Sigma)=0$ in $\mathbb{R}/\mathbb{Z}$ would imply that $\Sigma$ is diffeomorphic to $\mathbb{S}^7$, since Eells and Kuiper \cite{EK62} showed that $\mu(\Sigma_M^{\#m})=\frac{m}{28}$, where $\Sigma_M$ is a generator of $\Theta_7$. Thus $I(M)=0$.
\end{proof}
\begin{remark}\label{realine}\rm
By Lemma \ref{inerreal}, we can now prove the following.
\begin{itemize}
\item[(i)] If a closed smooth manifold $M$ is homotopy equivalent to $\mathbb{R}\textbf{P}^7$, then $M$ is a spin manifold with $H^4(M;\mathbb{R})=0$ and hence $I(M)=0$.
\item[(ii)] If $M$ is a closed $2$-connected $7$-manifold such that the group $H_4(M;\mathbb{Z})$ is torsion, then $M$ is a spin manifold with $H^4(M;\mathbb{R})=0$ and hence $I(M)=0$.
\end{itemize}
\end{remark}
By Lemma \ref{inerreal} and Remark \ref{unique}(ii), we have
\begin{theorem}\label{clasdiff}
Let $M$ be a closed smooth spin $7$-manifold such that $H^4(M;\mathbb{R})=0$ and $N$ be a closed smooth manifold PL-homeomorphic to $M$. Then there is a unique homotopy $7$-sphere $\Sigma^7\in \Theta_7$ such that $N$ is diffeomorphic to $M\#\Sigma^7$.
\end{theorem}
Applying Theorem \ref{clasdiff}, we immediately obtain
\begin{corollary}
Let $M$ be a closed smooth manifold homotopy equivalent to $\mathbb{R}\textbf{P}^7$. Then $M$ has, up to diffeomorphism, exactly $28$ distinct differentiable structures with the same underlying PL structure of $M$.
\end{corollary}
\begin{remark}\rm
If a closed smooth manifold $M$ is homotopy equivalent to $\mathbb{R}\textbf{P}^n$, where $n=5$ or $6$, then $M$ has exactly $2$ distinct differentiable structures up to diffeomorphism \cite{KS69, KS77,Hir63, HM74}.
\end{remark}
\section{The classification of smooth structures on a fake real projective space}
The following theorem was proved in \cite[Example 3.5.1]{Rud16} for $M=\mathbb{R}\textbf{P}^7$. This proof works verbatim for an arbitrary manifold.
\begin{theorem}\label{plexam}
Let $M$ be a closed smooth manifold homotopy equivalent to $\mathbb{R}\textbf{P}^7$. Then there is a closed smooth manifold $\widetilde{M}$ such that 
\begin{itemize}
\item [(i)] $\widetilde{M}$ is homeomorphic to $M$.
\item[(ii)] $\widetilde{M}$ is not (PL homeomorphic) diffeomorphic to $M$.
\end{itemize}
\end{theorem}
\begin{proof}
Let $j_{TOP}:\mathcal{C}^{PL}(M)\to [M, TOP/PL]=H^3(M;\mathbb{Z}_2)$ be a bijection given by \cite{KS77, KS69} and $j_{F}:\mathcal{S}^{PL}(M)\to [M, F/PL]$ be the normal invariant map defined by \cite{Sul67, Sul67a}. Then the maps $j_{TOP}$ and $j_{F}$ can be included in the commutative diagram
\begin{displaymath}\label{comm}
    \xymatrix@C=3.5cm{
        \mathcal{C}^{PL}(M) \ar[r]^{j_{TOP}} \ar[d]_{\mathcal{F}} & [M, TOP/PL] \ar[d]^{a_{*}} \\
        \mathcal{S}^{PL}(M) \ar[r]_{j_{F}}       & [M, F/PL] }
\end{displaymath}
where $\mathcal{F}$ is the obvious forgetful map and $a_{∗}$ is induced by the natural map $a:TOP/PL\to F/PL$. Consider an element $[\widetilde{M},k]\in  \mathcal{C}^{PL}(M)$, where $\widetilde{M}$ is a closed PL-manifold and $k:\widetilde{M}\to M$ is a homeomorphism such that 
\begin{equation}\label{map}
j_{TOP}([\widetilde{M},k])\neq 0\in [M, TOP/PL]=H^3(M;\mathbb{Z}_2)\cong \mathbb{Z}_2.
\end{equation}
 Notice that the Bockstein homomorphism $$\delta:\mathbb{Z}_2=H^3(M;\mathbb{Z}_2)\to H^4(M;\mathbb{Z}[2])=\mathbb{Z}_2$$ is an isomorphism, where $\mathbb{Z}[2]$ is the subring of $\mathbb{Q}$ consisting of all irreducible fractions with denominators relatively prime to $2$. Hence $$\delta(j_{TOP}([\widetilde{M},k]))\neq0.$$ So, by \cite[Corollary 3.2.5]{Rud16}, $a_*(j_{TOP}([\widetilde{M},k]))\neq 0$. In view of the above commutativity of the diagram, $$j_{F}( \mathcal{F}([\widetilde{M},k]))=a_*(j_{TOP}([\widetilde{M},k])),$$ i.e., $j_{F}( \mathcal{F}([\widetilde{M},k]))\neq 0$. This implies that $\mathcal{F}([\widetilde{M},k])\neq 0$. Hence $[\widetilde{M},k]\neq [M,Id]$ in $\mathcal{S}^{PL}(M)$. On the other hand, it follows from the obstruction theory that every orientation-preserving homotopy equivalence $h:M\to M$ is homotopic to the identity map. This shows that $\widetilde{M}$ is not PL homeomorphic to $M$. By an obstruction theory given by \cite{HM74}, every PL-manifold of dimension $7$ possesses a compatible differentiable structure. This implies that $\widetilde{M}$ is smoothable such that $\widetilde{M}$ can not be diffeomorphic to $M$. This proves the theorem.
\end{proof}
\begin{theorem}\label{diffclass}
Let $M$ be a closed smooth manifold homotopy equivalent to $\mathbb{R}\textbf{P}^7$. Then $$\mathcal{C}^{Diff}(M)=\left \{[M\#\Sigma, Id], [\widetilde{M}\#\Sigma, k\circ Id]~~|~~ \Sigma\in \Theta_{7} \right \},$$ where $\widetilde{M}$ is the
specific closed smooth manifold given by Theorem \ref{plexam} and $k:\widetilde{M}\to M$ is the homeomorphism as in Equation (\ref{map}). In particular, $M$ has exactly $56$ distinct differentiable structures up to concordance.
\end{theorem}
\begin{proof}
Let $[N,f]\in \mathcal{C}^{Diff}(M)$, where $N$ is a closed smooth manifold and $f:N\to M$ be a homeomorphism. Then $(N,f)$ represents an element in $$\mathcal{C}^{PL}(M)\cong H^3(M;\mathbb{Z}_2)=\mathbb{Z}_2=\left \{ [M,Id], [\widetilde{M}, k] \right \},$$ where $\widetilde{M}$ is the specific closed smooth manifold given by Theorem \ref{plexam} and $k:\widetilde{M}\to M$ be a homeomorphism as in Equation (\ref{map}). This implies that $(N,f)$ is either equivalent to $(M,Id)$ or $(\widetilde{M}, k)$ in $\mathcal{C}^{PL}(M)$. Suppose that $(N,f)$ is equivalent to $(M,Id)$ in $\mathcal{C}^{PL}(M)$, then there is a PL-homeomorphism $h:N\to M$ such that $Id\circ h:N\to M$ is topologically concordant to $f:N\to M$. Now consider a pair $(N,h)$ which represents an element in $\mathcal{C}^{PDiff}(M)$. By Theorem \ref{pdiff}, there is a unique homotopy sphere $\Sigma$ such that $(N,h)$ is PL-concordant to $(M\#\Sigma,Id)$. Hence there is a diffeomorphism $\phi:N\to M\#\Sigma$ such that $Id\circ \phi:N\to M$ is topologically concordant to $h:N\to M$. Note that $Id\circ h:N\to M$ is topologically concordant to $f:N\to M$. This implies that $Id\circ \phi:N\to M$ is topologically concordant to $f:N\to M$. Therefore, $(N,f)$ and $(M\#\Sigma,Id)$ represent the same element in $\mathcal{C}^{Diff}(M)$.\\
On the other hand, suppose that $(N,f)$ is equivalent to $(\widetilde{M},k)$ in $\mathcal{C}^{PL}(M)$. This implies that there is a PL-homeomorphism $h:N\to \widetilde{M}$ such that $k\circ h:N\to M$ is topologically concordant to $f:N\to M$. By using the same argument as above, we have that there is a unique homotopy sphere $\Sigma$ and a diffeomorphism $\phi:N\to \widetilde{M}\#\Sigma$ such that $$k\circ Id\circ \phi:N\to\widetilde{M}\#\Sigma\to \widetilde{M}\to M$$ is topologically concordant to $f:N\to M$. Therefore, $(N,f)$ and $(\widetilde{M}\#\Sigma, k\circ Id)$ represent the same element in $\mathcal{C}^{Diff}(M)$.\\
Thus, there is a unique homotopy sphere $\Sigma$ such that $(N,f)$ is either concordant to $(M\#\Sigma, Id)$ or $(\widetilde{M}\#\Sigma, k\circ Id)$ in $\mathcal{C}^{Diff}(M)$. This shows that $$\mathcal{C}^{Diff}(M)=\left \{[M\#\Sigma, Id], [\widetilde{M}\#\Sigma, k\circ Id]~~|~~ \Sigma\in \Theta_{7} \right \}.$$ In particular, $M$ has exactly $56$ distinct differentiable structures up to concordance.
\end{proof}
\begin{theorem}\label{free}
Let $M$ be a closed smooth manifold homotopy equivalent to $\mathbb{R}\textbf{P}^7$. Then $\Theta_7$ acts freely on $\mathcal{C}^{Diff}(M)$.
\end{theorem}
\begin{proof}
Suppose $[N\#\Sigma,f]=[N,f]$ in $\mathcal{C}^{Diff}(M)$. Then $N\#\Sigma\cong N$. Since by Theorem \ref{diffclass}, there is a homotopy sphere $\Sigma_1$ such that $N\cong \overline{M}\#\Sigma_1$, where $\overline{M}=M$ or $\widetilde{M}$. This implies that $$\overline{M}\#\Sigma_1\#\Sigma^{-1}\cong \overline{M}\#\Sigma_1$$ and hence $\Sigma_1\#\Sigma^{-1}\#\Sigma_1^{-1}\in I(\overline{M})$. But, by Remark \ref{realine}(i), $I(\overline{M})=0$. This shows that $\Sigma_1\#\Sigma^{-1}\#\Sigma_1^{-1}\cong \mathbb{S}^7$. Hence $\Sigma\cong \mathbb{S}^7$. This proves that $\Theta_7$ acts freely on $\mathcal{C}^{Diff}(M)$.
\end{proof}
\begin{remark}\rm
Let $M$ and $\widetilde{M}$ be as in Theorem \ref{diffclass}. Then $\Theta_7$ does not act transitively on $\mathcal{C}^{Diff}(M)$, since $M$ and $\widetilde{M}$ are not PL-homeomorphic.
\end{remark}
\begin{theorem}
Let $M$ be a closed smooth manifold which is homotopy equivalent to $\mathbb{R}\textbf{P}^7$. Then $M$ has exactly $56$ distinct differentiable structures up to diffeomorphism. Moreover, if $N$ is a closed smooth manifold homeomorphic to $M$, then there is a unique homotopy sphere $\Sigma\in \Theta_7$ such that $N$ is either diffeomorphic to $M\#\Sigma$ or $\widetilde{M}\#\Sigma$, where $\widetilde{M}$ is the specific closed smooth manifold given by Theorem \ref{plexam}.
\end{theorem}
\begin{proof}
Let $N$ be a closed smooth manifold homeomorphic to $M$ and let $f:N\to M$ be a homeomorphism. Then $(N,f)$ represents an element in $\mathcal{C}^{Diff}(M)$. By Theorem \ref{diffclass}, there is a unique homotopy sphere $\Sigma\in \Theta_7$ such that $N$ is either concordant to $(M\#\Sigma, Id)$ or $(\widetilde{M}\#\Sigma, k\circ Id)$. This implies that $N$ is either diffeomorphic to $M\#\Sigma$ or $\widetilde{M}\#\Sigma$. By Remark \ref{realine}(i), $I(M)=I(\widetilde{M})=0$. Therefore there is a unique homotopy sphere $\Sigma\in \Theta_7$ such that $N$ is either diffeomorphic to $M\#\Sigma$ or $\widetilde{M}\#\Sigma$. This implies that $M$ has exactly $56$ distinct differentiable structures up to diffeomorphism.
\end{proof}

\end{document}